\documentclass[12pt]{amsart}
\usepackage{amsfonts, amsbsy, amsmath, amssymb}

\usepackage[T1]{fontenc}
\usepackage[utf8]{inputenc}
\usepackage{lmodern}
\usepackage{amsmath, amsthm, amssymb,amscd, mathrsfs, amsfonts, mathtools,tikz-cd}
\usepackage{relsize}
\usepackage{euler}
\usepackage{times}
\usepackage[all]{xy}
\usepackage{todonotes}
\usepackage{xcolor}
\usepackage[bookmarks=true,
            bookmarksnumbered=true, breaklinks=true,
            pdfstartview=FitH, hyperfigures=false,
            plainpages=false, naturalnames=true,
            colorlinks=true,pagebackref=true,
            pdfpagelabels]{hyperref}

\usepackage[lite]{amsrefs}

\renewcommand{\PrintDOI}[1]{\href{http://dx.doi.org/\detokenize{#1}}{doi: \detokenize{#1}}%
  \IfEmptyBibField{pages}{, (to appear in print)}{}}

\def\commutatif{\ar@{}[rd]|{\circlearrowleft}}

\newtheorem{thm}{Theorem}[section]
\newtheorem{pro}[thm]{Proposition}

\newtheorem{cor}[thm]{Corollary}
\newtheorem{conj}[thm]{Conjecture}

\theoremstyle{definition}
\newtheorem{df}[thm]{Definition}

\newtheorem{qst}[thm]{Question}

\theoremstyle{remark}
\newtheorem{rmk}[thm]{Remark}

\newtheorem{ex}[thm]{Example}

\allowdisplaybreaks

\newcommand{\R}{\mathbb{R}}

\newcommand{\Z}{\mathbb{Z}}

\newcommand\rt{\triangleright}

\usepackage{bbm}

\def\a{\alpha}

\hypersetup{
  colorlinks = true,
  urlcolor = blue,
  linkcolor = blue,
  citecolor = red,
  pdfauthor = {Elhamdadi, M., Fernando, N.},
  pdfkeywords = {Quandles, ring quandle},
  pdftitle = {Elhamdadi, M., Fernando, N. - Ring Theoretic Aspects of Quandles},
  pdfsubject = { racks, quandles},
  pdfpagemode = UseNone
}

\title{Ring Theoretic Aspects of Quandles}

\author{Mohamed Elhamdadi} 
\address{Department of Mathematics, 
University of South Florida, Tampa, FL 33620} 
\email{emohamed@math.usf.edu} 

\author{Neranga Fernando} 
\address{Department of Mathematics, 
Northeastern University, Boston MA 02115} 
\email{w.fernando@northeastern.edu} 

\author{Boris Tsvelikhovskiy} 
\address{Department of Mathematics, 
	Northeastern University, Boston MA 02115} 
\email{tsvelikhovskiy.b@husky.neu.edu}

\begin{document}

\maketitle

\begin{abstract}
We associate to every quandle $X$ and an associative ring with
unity $\mathbf{k}$, a \textit{nonassociative} ring $\mathbf{k}[X]$ following \cite{BPS}. The basic properties of such rings are investigated.
In particular, under the assumption that the inner automorphism group $Inn(X)$ acts \textit{orbit $2$-transitively} on $X$,
a complete description of right (or left) ideals is provided. The complete description of right ideals for the dihedral quandles $R_n$
is given. 
It is also shown that if for two quandles $X$ and $Y$ the inner automorphism groups act $2$-transitively and $\mathbf{k}[X]$ is isomorphic to $\mathbf{k}[Y]$,
then the quandles are of the same partition type. However, we provide examples when the quandle rings $\mathbf{k}[X]$ and $\mathbf{k}[Y]$ are isomorphic, but the quandles $X$ and $Y$ are not isomorphic. 
These examples answer some open problems in \cite{BPS}. 
 \end{abstract}

\tableofcontents

\section{Introduction}
Quandles are generally non-associative algebraic structures (the exception being the trivial quandles).  They were introduced independently in the 1980's by Joyce \cite{Joyce} and Matveev \cite{Matveev} with the purpose of constructing invariants of knots in the three space and knotted surfaces in four space.  However, the notion of a quandle can be traced back to the 1940's in the work of Mituhisa Takasaki \cite{Takasaki}.  The three axioms of a quandle algebraically encode the three Reidemeister moves in classical knot theory.  For a recent treatment of quandles (see \cite{EN}).  Joyce and Matveev introduced the notion of the fundamental quandle of a knot and gave a theorem that brings the problem of equivalence of knots to the problem of the quandle isomorphism of their fundamental quandles.  Precisely, two knots $K_1$ and $K_2$ are equivalent (up to reverse and mirror image) if and only if the fundamental quandles $Q(K_1)$ and  $Q(K_2)$ are isomorphic.  But determining isomorphism classes of quandles is a difficult task in general.  Thus the need of restricting oneself to some specific families of quandles such as connected quandles (called also indecomposable), medial and Alexander quandles.  Recall that the fundamental quandles of knots are connected.
Recently, there has been investigations of quandles from algebraic point of views and their relations to other algebraic structures such as Lie algebras \cite{CCES1, CCES2}, Leibniz algebras \cite{Kinyon, KW}, Frobenius algebras and Yang-Baxter equation \cite{CCEKS}, Hopf algebras \cite{AG, CCES2}, transitive groups \cite{Vend}, quasigroups and Moufang loops \cite{Elhamdadi}, ring theory \cite{BPS} etc.  This article will add to this list since we introduce new concepts motivated by ring theory to the theory of quandles.  We follow \cite{BPS} and we associate to every quandle $(X, \rt)$ and an associative ring $\mathbf{k}$ with unity, a \textit{nonassociative} ring $\mathbf{k}[X]$.   Precisely, let $\mathbf{k}[X]$ be the set of elements that are uniquely expressible in the form $\sum_{x \in X }  a_x x$, where $x \in X$ and $a_x=0$ for almost all $x$.  Then the  set $\mathbf{k}[X]$ becomes a ring with the natural addition and the multiplication given by the following, where $x, y \in X$ and $a_x, a_y \in \mathbf{k}$,
\[   ( \sum_{x \in X }  a_x x) \cdot ( \sum_{ y \in X }  b_y y )
=   \sum_{x, y \in X } a_x b_y (x \rt y) . \]
Linearization of quandles appeared in the work on categorical groups and other notions of
categorification in \cite{CCES1} and \cite{CCES2}, where self-distributive structures in the categories of coalgebras, cocommutative coalgebras and Hopf algebras were studied.  Precisely, in studying self-distributivity maps in coalgebras, the authors of \cite{CCES1} gave a broad examples with a focus on the case of $\mathbf{k} \oplus \mathbf{k}[X]$ with the multiplication \[ ( a+ \sum_{x \in X }  a_x x) \cdot ( b+\sum_{ y \in X }  b_y y )
= \sum_{y \in X } a b_y  +  \sum_{x, y \in X } a_x b_y (x \rt y) .\]

In \cite{BPS}, the authors showed that the ring $\mathbf{k}[X]$ gives interesting information on the quandle $X$.  In this article we investigate the basic properties of quandle rings and also solve some of the open problems stated in \cite{BPS}.   In particular, under the assumption that the inner automorphism group $Inn(X)$ acts orbit $2$-transitively on the quandle $X$ and the ring $\mathbf{k}$ is a field of characteristic zero (or a certain semigroup $H_x$ acts $2$-transitively on $X$) a complete description of right (or left) ideals is provided. The corresponding results for fields of positive characteristic are given in Corollary~\ref{RightModDecomp}. The complete description of right ideals of $\mathbf{k}[R_n]$, where $R_n$ is the dihedral quandle of order $n$,  is given. It is also shown that the rings $\mathbf{k}[X]$ are Noetherian, when the quandle $X$ is finite and the ring $\mathbf{k}$ is Noetherian.  We also give an example of a quandle $X$ with $\mathbf{k}[X]$ not Noetherian. These rings are, in general, not domains and neither every right nor left ideal is principal. It is also shown that if for two quandles $X$ and $Y$ the inner automorphism groups act $2$-transitively and $\mathbf{k}[X]$ is isomorphic to $\mathbf{k}[Y]$  (here $\mathbf{k}$ is a field of $char=0$), then the quandles are of the same partition type. However, we provide examples when the quandle rings $\mathbf{k}[X]$ and $\mathbf{k}[Y]$ are isomorphic, but the quandles $X$ and $Y$ are not isomorphic with $\mathbf{k})$ being a field of any charcteristic.

The following is the organization of the article.  In Section~\ref{review}, we recall the basics of quandles with examples.  In Section~\ref{Powerassoc}, we investigate an open question raised in \cite{BPS} concerning the power associativity of non-trivial quandles. Precisely, we prove that quandle rings are never power associative when the quandle is non-trivial and $char(\mathbf{k})\neq 2, 3$.

Section~\ref{properties} deals with various properties of quandle rings.  We show that for a Noetherian ring $\mathbf{k}$ and a finite quandle $X$, the quandle ring $\mathbf{k}[X]$ is both left and right Noetherian ring.  We also give, for any positive integer $n$ and $\mathbf{k}=\mathbb{R} \mbox{ or } \mathbb{C}$, the complete list of simple right ideals of the quandle ring $\mathbf{k}[R_n]$.   In section~\ref{iso}, we investigate the problem of isomorphisms of quandle rings.  We introduce the notion of partition type of quandles and show that if the quandle rings $\mathbf{k}[X]$ and $\mathbf{k}[Y]$ are isomorphic and the quandles $X$ and $Y$ are orbit $2$-transitive, then $X$ and $Y$ are of the same partition type.  Section~\ref{augmentation} deals with the augmentation ideals of  quandle rings.  Precisely we give a solution to conjecture 6.5 in \cite{BPS}.

Throughout the paper, $\mathbf{k}$ always denotes a ring unless specified otherwise. Also, quandle operation and ring operation are denoted by $\rt$ and $\cdot$, respectively. 

\section{Review of Quandles}\label{review}
We start this section by giving the basics of quandles with examples.
\begin{df}\label{quandledef}
A {\it quandle}, $X$, is a set with a binary operation $(a, b) \mapsto  a \rt b$ such that

(I) For any $a \in X$,
$a\rt a =a$.

(II) For any $a,b \in X$, there is a unique $c \in X$ such that
$a= c\rt b$.

(III)
For any $a,b,c \in X$, we have
$ (a \rt b) \rt c=(a\rt c)\rt (b\rt c). $
\end{df}

\noindent A quandle $(X,\rt)$ is said to be commutative if $a\rt b=b\rt a, \forall a,b \in X$.

\noindent A {\it rack} is a set with a binary operation that
satisfies (II) and (III). Racks and quandles have been studied
extensively in, for example, \cite{Joyce,Matveev}.  For more details on racks and quandles see the book \cite{EN}.

The following are typical examples of quandles: 
\begin{itemize}
	
	\item
	Let $X$ be a non empty set.  The binary operation $a \rt b= a, \; \forall a,b \in X,$ defines a quandle operation on $X$ called \emph{trivial quandle.}
\item
A group $G$ with
conjugation as the quandle operation: $a \rt b = b^{-1} a b$,
denoted by $X=$ Conj$(G)$, is a quandle. 

\item
Any subset of $G$ that is closed under such conjugation is also a quandle. More generally if
$G$ is a group, $H$ is a subgroup, and $\sigma$ is an automorphism that
fixes the elements of $H$  ({\it i.e.} $\sigma(h)=h \ \forall h \in
H$), then $G/H$ is a quandle with $\rt $ defined by $Ha\rt
Hb=H \sigma(ab^{-1})b.$ 

\item
Any ${\Z }[t, t^{-1}]$-module $M$ is
a quandle with $a\rt b=ta+(1-t)b$, for $a,b \in M$, and is called
an {\it  Alexander  quandle}. 

\item
Let $n$ be a positive integer, and
for elements  $i, j \in \Z_n$, define $i\rt j = 2j-i \pmod{n}$. Then $\rt$ defines a quandle structure
called the {\it dihedral quandle}, and denoted by $R_n$, that
coincides with the set  of reflections in the dihedral group
with composition given by conjugation.

\item
Any group $G$ with the quandle operation: $a \rt b = ba^{-1}  b$ is a quandle called Core(G). 
\end{itemize}

The notions of quandle homomorphims and automorphisms are clear.  Let $X$ be a quandle, thus the second axiom of Definition~\ref{quandledef} makes any right multiplication by any element $x$, $R_x: y \mapsto y \rt x$, into a bijection .  The third axiom of Definition~\ref{quandledef} makes $R_x$ into a homomorphism and thus an automorphism.  Let $Aut(X)$ denotes the group of all automorphisms of $X$ and let $Inn(X):=<R_x, \; x \in X>$ denotes the subgroup generated by right multiplications.  The quandle $X$ is called \textit{connected} quandle if the group $Inn(X)$ acts transitively on $X$, that is, there is only one orbit.  Later in the paper, in Section~\ref{properties}, we will use the left multiplication in a quandle denoted $L_x: y \mapsto x \rt y$.  In general these maps need not to be bijective.  Quandles in which left multiplications $L_x$ are bijections are called \textit{Latin} quandles.

\section{Power associativity of quandle rings}\label{Powerassoc}
 In \cite{BPS}, power associativity of dihedral quandles was investigated and the question of determining the conditions under which the  quandle ring $R[X]$ is power associative was raised.  In this section we give a complete solution to this question.  Precisely, we prove that quandle rings are never power associative when the quandle is non-trivial and $char(\mathbf{k})\neq 2, 3$.	
  
But first let's recall the following definition from \cite{Albert}.
\begin{df}
	A ring $\mathbf{k}$ in which every element generates an associative subring is called a \textit{power-associative} ring. 
\end{df}

\begin{ex}{\rm
		Any alternative algebra  is power associative.  Recall that an algebra $A$ is called \textit{alternative} if $x \cdot (x \cdot y)= (x \cdot x )\cdot y$ and  $x \cdot (y \cdot y)= (x \cdot y )\cdot y, \forall x, y \in A$,  (for more details see \cite{EM}).  
	} \end{ex}

		It is well known \cite{Albert} that a ring $\mathbf{k}$ of characteristic zero is power-associative if and only if 
		\[(x \cdot x)\cdot x=x\cdot (x \cdot x) \; \textit{and}\; (x \cdot x)\cdot (x \cdot x)=[(x \cdot x)\cdot x] \cdot x,\; \textit{for all}\; x \in \mathbf{k}. \]

\begin{rmk}\label{Erandi1}
It follows from (II) and (III) of Definition~\ref{quandledef} that a quandle $(X,\rt)$ is associative if and only it it is trivial. 
\end{rmk}

\begin{thm}
Let $\mathbf{k}$ be a ring with $char(\mathbf{k})\neq 2, 3$ and $(X,\rt)$ be a non-trivial quandle. Then the quandle ring $\mathbf{k}[X]$ is not power associative.
\end{thm}

\begin{proof}
Clearly, if $(X,\rt)$ is associative, then the quandle ring  $\mathbf{k}[X]$ is power associative. Hence we assume that $(X,\rt)$ is not associative (equivalently not trivial, see Remark~\ref{Erandi1}). Therefore, henceforth in the proof we assume $(X,\rt)$ is a non-trivial quandle.

We prove the following:  
\begin{enumerate}
\item [(a)] If there exist $x, y\in X$ such that $x\neq y$ and $x\rt y = y\rt x$, then $\mathbf{k}[X]$ is not power associative.
\item [(b)] Assume  $ \forall x, y\in X$ such that $x\neq y$, we have  $x\rt y \neq y\rt x$. If $\mathbf{k}[X]$ is power associative, then $x\rt y= x$.
\end{enumerate}

Note that (a) and (b) imply that $(X,\rt)$ is trivial, which contradicts our assumption. 

\par

Choose $x,y \in X$ such that $x\neq y$  and $x\rt y = y\rt x$. Let $u=ax+by$, where $ab \neq 0$. Then 
$$u\cdot u=a^2x+2ab\,(x\rt y)+ b^2y.$$
It is straightforward to see that $(u\cdot u)\cdot u = u\cdot (u\cdot u)$. Now we consider $(u\cdot u)\cdot (u\cdot u)$ and $((u\cdot u)\cdot u)\cdot u$. We have
\begin{equation}\label{NK1}
(u\cdot u)\cdot (u\cdot u)=x\,a^4+4(x\rt (x\rt y))\,a^3b+4(y\rt(x\rt y))\,ab^3+6(x\rt y)\,a^2b^2+y\,b^4
\end{equation}
and 
\begin{equation}\label{NK2}
\begin{split}
((u\cdot u)\cdot u)\cdot u&=x\,a^4+b^4\, y +\big[x\rt y + (x\rt y)\rt x + 2[(x\rt y)\rt x]\rt x\big]\,a^3b\cr
&+\big[y\rt x + (y\rt x)\rt y + 2[(x\rt y)\rt y]\rt y\big]\,ab^3\cr
&+\big[(x\rt y)\rt y + 2[(x\rt y)\rt x]\rt y + 2[(x\rt y)\rt y]\rt x + (y\rt x)\rt x\big]\,a^2b^2
\end{split}
\end{equation}

Assume to the contrary that $\mathbf{k}[X]$ is power associative. Then $(u\cdot u)\cdot (u\cdot u) = ((u\cdot u)\cdot u)\cdot u$. Since $char(\mathbf{k})\neq 2, 3$, it is direct to check (plugging in different values for $a$ and $b$) that the coefficients (which are the quandle ring elements) of $ab^3$, $a^3b$ and $a^2b^2$ must be pairwise equal on both sides of the equality $(u\cdot u)\cdot (u\cdot u) = ((u\cdot u)\cdot u)\cdot u$. In particular, by comparing the coefficients of $a^2b^2$ in \eqref{NK1} and \eqref{NK2}, we get $x\rt y =(x\rt y)\rt y$, which implies $x = x\rt y$. Since $x\rt y = y\rt x$, $x = x\rt y$ implies $x=y$ which contradicts our assumption that $x\neq y$. This completes the proof of part (a). 

Next we prove part (b). Choose $x, y\in X$ such that $x\neq y$ and $x\rt y \neq y\rt x$. Let $u=ax+by$, where $ab$ is nonzero. Then 
$$u\cdot u=a^2x+ab\,(x\rt y)+ ab (y\rt x)+b^2y.$$

Now we consider $(u\cdot u)\cdot u$ and $u\cdot (u\cdot u)$. We have
\begin{equation}\label{NK3}
\begin{split}
(u\cdot u)\cdot u &= x\,a^3+ y\,b^3 + \big[(x\rt y) + (x\rt y)\rt x + (y\rt x)\rt x \big]\,a^2b\cr
&+\big[(y\rt x) + (x\rt y)\rt y + (y\rt x)\rt y \big]\,ab^2
\end{split}
\end{equation}

and 

\begin{equation}\label{NK4}
\begin{split}
u\cdot (u\cdot u)&= x\,a^3+ y\,b^3 + \big[(y\rt x) + x\rt (x\rt y) + x\rt (y\rt x)\big]\,a^2b\cr
&+\big[(x\rt y) + y\rt (x\rt y) + y\rt (y\rt x) \big]\,ab^2.
\end{split}
\end{equation}

Assume that $\mathbf{k}[X]$ is power associative. Then we must have $(u\cdot u)\cdot u = u\cdot (u\cdot u)$. Since $char(\mathbf{k})\neq 2$, by setting \eqref{NK3}=\eqref{NK4} and letting $a=b=1$ and $a=-1,b=1$, we get 
$$x\rt y + (x\rt y)\rt x + (y\rt x)\rt x = y\rt x + x\rt (x\rt y) + x\rt (y\rt x).$$

Since $(x\rt y)\rt x=x\rt (y\rt x)$, the equation becomes 
$$x\rt y + (y\rt x)\rt x = y\rt x + x\rt (x\rt y).$$

Since $x\rt y \neq y\rt x$, we get $x\rt (x\rt y)= x\rt y$, which implies $x\rt y = x$. This completes the proof of part (b). 

\end{proof}		
\section{Various properties of the quandle ring $\mathbf{k}[X]$} \label{properties}
In this section, we investigate different properties of quandle rings.  
\subsection{Basic properties}

The following proposition shows that if the quandle $X$ is a union of a finite orbit $X_1$ with more than one element and any quandle $X_2$, then the quandle ring $\mathbf{k}[X]$ is not an integral domain.
	
\begin{pro}
Let $X=X_1\ \amalg X_2$ be a quandle with $1<|X_1|<\infty$. Then $\mathbf{k}[X]$ is not a domain.
\end{pro}

\begin{proof}
Indeed, it follows from property $(II)$ of Definition~\ref{quandledef} that  $$(\sum\limits_{z\in X_1} z)\cdot (x-y)=0,$$ where $x$ and $y$ are any two distinct elements of $X$.
\end{proof}

\begin{rmk}
If the ring $\mathbf{k}$ is a domain and $X=\{x\}$ is the one element quandle, then  $\mathbf{k}[X]$ is a domain as well.
\end{rmk}
\begin{qst}
Are there any other quandles  $X$, for which $\mathbf{k}[X]$ is a domain? 
\end{qst}

\begin{df}
Let $X=\{e_1, \cdots, e_n\}$ be a quandle of finite cardinality. We define the \textit{order} of an element $x\in \mathbf{k}[X]$ to be the largest index $i$ that occurs in the expression $x=\sum a_ie_i$.
\end{df}
We then have the following proposition that gives the conditions for the quandle ring $\mathbf{k}[X]$ to be both left and right Noetherian.
\begin{pro}
Let $X$ be a quandle of finite cardinality and $\mathbf{k}$ be a Noetherian ring. Then $\mathbf{k}[X]$ is both left and right Noetherian ring.
\end{pro}
\begin{proof}
 Let $I\subset \mathbf{k}[X]$ be an ideal, $I_m$  the subset  of elements of order $m$ in $I$ and $\tilde{I}_m\subset \mathbf{k}$  the ideal of leading coefficients of elements in $I_m$. Clearly, $I=\underset{m}{\cup} I_m$, moreover, each $\tilde{I}_m$ is finitely generated, since  the ring $\mathbf{k}$ is Noetherian. Now it is sufficient to verify that $I_m$ is generated by any $(f_{m_1},\hdots,f_{m_k})\subset \overset{m}{\underset{s=1}{\cup}} I_s$, whose leading coefficients $(a_{m_1},\hdots,a_{m_k})$ generate $\overset{m}{\underset{s=1}{\cup}}\tilde{I}_s$. This is checked via induction on $m$. Indeed, let $g=\alpha_me_m+\sum\limits_{i\leq m-1} a_ie_i\in I_m$ be an element of order $m$. Then $\alpha_m=\sum\limits_{i=1}^{k}\beta_ia_{m_i}$ and the element  $g-\sum\limits_{i=1}^{k}\beta_if_{m_i}$ has order strictly less than $m$.

\end{proof}
\par
The next example shows that $\mathbf{k}[X]$ is neither necessarily a right nor a left Noetherian ring, if the quandle $X$ is not of finite cardinality.
\begin{ex}
Consider the trivial quandle $X= X_0\amalg X_1\amalg X_2\amalg X_3\amalg \hdots $ with each $X_i=\{e_i\}$ consisting of a single element and $\mathbf{k}$ a field. 

	Take the chain $I_1\subset I_2 \subset I_3 \subset \hdots$ with $I_j:=\mathbf{k}\langle e_1-e_0, e_2-e_0, \hdots, e_j-e_0\rangle$. As for any $\gamma \in \mathbf{k}[X]$ and any $a \in I$, where $I \subset \mathbf{k}[X]$ is the augmentation ideal, we have $\gamma\rt a=0$ and $a\rt\gamma=ca$ for some $c\in\mathbf{k}$, each $I_j$ is indeed an ideal and the chain $I_1\subset I_2 \subset I_3 \subset \hdots$  (considered as left or right or two-sided ideals) does not stabilize.

\end{ex}

\par
A similar example shows that the ring $\mathbf{k}[X]$ is not necessarily principal, even if the quandle $X$ is of finite cardinality.	\begin{ex}
Consider the trivial quandle $X=X_1\amalg X_2\amalg\hdots\amalg X_k$ with each $X_i=\{e_i\}$ consisting of a single element and let $I \subset \mathbf{k}[X]$ be the augmentation ideal. For the same reasons as in the example above, the ideal $I$ can not be generated by a single element if $n\geq 3$.

\end{ex}

\subsection{Study of ideals in $\mathbf{k}[X]$ via groups and semigroups}
Let $\mathcal{T}_X$ stand for the semigroup of all maps from a finite set $X$ to itself (the full transformation semigroup of $X$). Representations of $\mathcal{T}_X$ were extensively studied by A. H. Clifford (see for example \cite{HZ}) and are intrinsically connected to the study of left ideals in $\mathbf{k}[X]$ as shown below.
\par
Let $X=X_1\amalg X_2\amalg\hdots\amalg X_k$ be the decomposition of a quandle $X$ into orbits. Recall that any element $x \in X$ gives rise to two functions from $X$ to itself given by $L_x: X\rightarrow X$ and $R_x: X\rightarrow X$ (the latter function is a bijection). The left multiplication produces a map  $\psi: X\rightarrow\mathcal{T}_X $. The function $R_x: X\rightarrow X$ restricts to functions $X_i\rightarrow X_i$ for any orbit $X_i$ by definition. This gives rise to the map $\varphi: X\rightarrow S_{n}$, which restricts to the maps $\varphi_i: X\rightarrow S_{n_i}$, where $n_i=|X_i|$ and $n=\sum\limits_{i=1}^{k} n_i$. The composition of functions  
$L_y \circ L_x$ and  
$R_y \circ R_x: X\rightarrow X$ correspond to the products of the (semi)group elements in the image. We denote the semigroup generated by the image of $\psi  (X)$ by $H_X\subset S_X$, the semigroup generated by the image of  $\varphi  (X)$ by $Inn(X)\subset S_X$ (this group is  known as the group of inner automorphisms of $X$ (see \cite{EMR})) and the groups generated by the images of  $\varphi_i  (X)\subset S_{n_i}$'s by $G_{X_i}\subset S_{n_i}$.  However, $H_X$ is not necessarily a group. For example, if the quandle $X$ contains a single element orbit  $\{x\}$ then $H_X$ contains the map sending all elements of $X$ to $ x$, which is not invertible.
\par
We denote the underlying $\mathbf{k}$ - vector space of $\mathbf{k}[X]$ by $V$. By remarks of the preceding paragraph to study the right ideals in 
$\mathbf{k}[X]$ is equivalent to viewing $V$ as a representation of the group 
$Inn(X)$. Clearly, $V=\underset{i=1}{\overset{k}{\oplus}}V^i$, where $V^i= \mathbf{k}[X_i]$, and therefore, one needs to understand the decompositions of the  $V^i$'s into irreducible representations of corresponding $G_{X_i}$'s. First, notice that each $V^i$ contains a one-dimensional subspace $V_{triv}:=\mathbf{k}v_{triv}$ invariant under the action of $G_{X_i}$, where $v_{triv}=\sum\limits_{x\in X_i} x$.

The following notion of \textit{$2$-transitivity} (and also higher $k$-transitivity) of quandles was introduced in \cite{McCarron2, McCarron}.  Two-transitive quandles are called \textit{two-point homogeneous} quandles in \cite{Tamaru}.
\begin{df}
The action of a group (semigroup) $G$ on a set $X$ is called \textit{$2$-transitive} if for any two pairs $(x_1,y_1)\in X\times X$ and $(x_2,y_2)\in X\times X$, there exists an element $g\in G$, such that $g\cdot x_1 =x_2$ and $g\cdot y_1 =y_2$.
\end{df}
\begin{df}
Let $X=X_1\amalg X_2\amalg\hdots\amalg X_k$ be a finite quandle. $X$ is said to be \textit{left (right) $2$-transitive} if the semigroup  $H_X$ (the group $Inn(X)$) acts $2$-transitively on $X$. $X$ is said to be \textit{left (right) orbit $2$-transitive} if the semigroup  $H_X$ (each of the groups $G_{X_i}$) acts $2$-transitively on $X$ (orbit $X_i$).
\end{df}
\par
Let $V_{st}\subset V$ be the subspace orthogonal to the vector $v_{triv}=\sum\limits_{x\in X_i} x$. 

The following theorem gives the list of subgroups $G\leqslant  S_n$ for which the representation $V_{st}$ is irreducible. The first assertion can be found as Theorem $1 (b)$ in \cite{Saxl} and second as $(ii)$ in the Main Theorem of \cite{BrKl}, and the proof and description of groups $(i)-(v)$ can be found in \cite{Mortimer}. The groups $(i)-(v)$ are respectively the affine and projective general semilinear groups, projective semilinear unitary groups, Suzuki and Ree groups. We refer the reader to \cite[Chapter 7]{DixonMortiner} for definitions of these groups. 
\par
\begin{thm} 
\label{TransClassif}
\begin{enumerate}
\item[(a)]

 If $char(\mathbf{k})=0$  then $V_{st}$ is an irreducible representation of the subgroup $G < S_n$ if and only if  $G$ is $2$-transitive and $A_n\not\leq G$. 
 \item[(b)] 
 In case $char(\mathbf{k})=p>3$  then $V_{st}$ is an irreducible representation of the subgroup $G < S_n$ if and only if  $G$ is $2$-transitive and $A_n\not\leq G$ except 
 \begin{enumerate}
 \item[(i)] $G \leqslant  A\Gamma L(m, q)$, and $p$ divides $q$;
 \item[(ii)] $G \leqslant  P\Gamma L(m, q)$, $m\geq 3$ and $p$ divides $q$;
 \item[(iii)] $G \leqslant  P\Gamma U(3, q)$, and $p$ divides $q+1$;
 \item[(iv)] $G \leqslant  Sz(q)$, and $p$ divides $q+1+m$, where $m^2=2q$;
 \item[(v)] $G \leqslant  Re(q)$, and $p$ divides $(q+1)(q+1+m)$, where $m^2=3q$.
\end{enumerate} 
\end{enumerate}
\end{thm}

\begin{df}
Let $S$ be a finite semigroup. If $e$ is
an idempotent, then $eSe$ is a monoid with identity $e$; its group of units $G_e$ is called
the \textit{maximal subgroup} of $S$ at $e$.
\end{df}
\begin{cor}\label{RightModDecomp}
\begin{enumerate}
		\item 
Let $X=X_1\amalg X_2\amalg\hdots\amalg X_k$ be a finite quandle with a $2$-transitive action of each $G_{X_i}$  on the corresponding orbit $X_i$ and $char(\mathbf{k})=0$. Then $\mathbf{k}[X]\simeq \underset{i=1}{\overset{k}{\bigoplus}}(V^i_{st}\oplus V^i_{triv})$ where the r.h.s. consists of simple right ideals.
\item 
Let $X=X_1\amalg X_2\amalg\hdots\amalg X_k$ be a finite quandle with a $2$-transitive action of each $G_{X_i}$  on the corresponding orbit $X_i$  and $G_{X_i}$'s not among the groups from $(i)-(v)$ in Theorem~\ref{TransClassif},  $char(\mathbf{k})=p>3$. Then $\mathbf{k}[X]\simeq \underset{i=1}{\overset{k}{\bigoplus}}(V^i_{st}\oplus V^i_{triv})$ where the r.h.s. consists of simple right ideals.
\par
\item
Let $X$ be a finite quandle such that  $H_X\subset \mathcal{T}_X$ contains a maximal subgroup with a $2$-transitive action on $X$ and $char(\mathbf{k})=0$. Then $\mathbf{k}[X]\simeq V_{st}\oplus V_{triv}$ where the r.h.s. consists of simple left ideals.
\par
\item
Let $X$ be a finite quandle such that  $H_X\subset \mathcal{T}_X$ contains a maximal subgroup with a $2$-transitive action on $X$ not among the groups from $(i)-(v)$ in Theorem~\ref{TransClassif} and $char(\mathbf{k})=p>3$. Then $\mathbf{k}[X]\simeq V_{st}\oplus V_{triv}$ where the r.h.s. consists of simple left ideals.\end{enumerate}
\end{cor}

\begin{df}
A quandle $X$ is of \textit{right (left) cyclic type (or cyclic)} if
for each $x \in X$ the permutation $\varphi (x) (\mbox{or } \mbox{ the semigroup element } \psi(x))$ acts on $X \setminus \{x\}$ as a cycle of length $|X| - 1$, where
$|X|$ denotes the cardinality of X. 
\end{df}
\par
The above discussion motivates to find the conditions on $X$ to be $2$-transitive and orbit $2$-transitive. The following theorem was obtained in \cite{Vend} (see Corollary $4$ and references therein). 
\begin{thm}
Every finite right  $2$-transitive quandle is of right cyclic type.
\end{thm}
\begin{rmk}
It is also easy to see that if $X$ is of left cyclic type then it is a left $2$-transitive quandle.
\end{rmk}
\par
The above observations allow to strengthen Corollary~\ref{RightModDecomp}.
\begin{cor}
\begin{enumerate}
\item 
Let $X=X_1\amalg X_2\amalg\hdots\amalg X_k$ be a finite quandle with each subquandle $X_i$ of right cyclic type. Then $\mathbf{k}[X]\simeq \underset{i=1}{\overset{k}{\bigoplus}}(V^i_{st}\oplus V^i_{triv})$,  where the r.h.s. consists of simple right ideals.
\par
\item
Let $X$ be a finite quandle of left cyclic type. Then $\mathbf{k}[X]\simeq V_{st}\oplus V_{triv}$ (an equality of left ideals with the r.h.s. consisting of simple left ideals).
\end{enumerate}

\end{cor}
\par

For all quandles with  order  up to eight, the following chart gives the number of quandles (up to isomorphism) and their corresponding number of right $2$-transitive quandles.

\begin{table}[ht]
\begin{center}
\begin{tabular}{ |c|c|c|c| } 
 \hline
Order & \# of quandles & \# of right $2$-transitive quandles &  \# of left $2$-transitive quandles  \\ 
  \hline
 $3$ & $3$ & $3$ & $2$ \\ 
  \hline
 $4$ & $7$ & $6$  & $3$  \\ 
 \hline
$5$ &  $22$ & $16$  & $7$ \\ 
 \hline
$6$ &  $73$ & $42$  & $14$ \\ 
 \hline
$7$ &  $298$ & $151$   & $39$\\ 
 \hline
$8$ &  $1581$ & $656$  & $105$ \\ 
 \hline
\end{tabular}
\end{center}
\caption{Right $2$-transitive quandles}
\end{table}

\par
\begin{ex}
Consider the dihedral quandle $R_n$ with odd $n>3$. This quandle is connected and the action of $Inn(R_n)$ is not $2$-transitive.
\end{ex}
\par 
The next proposition provides the decomposition of $\mathbf{k}[R_n]$ into the sum of simple right ideals. In the following proposition, $D_n$ denotes the group of symmetries of a regular $n$-gon. 

\begin{pro}
Fix the ground field  $\mathbf{k}=\mathbb{C}, \mathbb{R}$.
\begin{enumerate}
		\item 
		Let $n$ be an odd number. Then $\mathbf{k}[R_n]\simeq V_{triv}\oplus\underset{\xi\in \mu_n}{\bigoplus} V_{\xi, \overline{\xi}}$  where the r.h.s. consists of simple right ideals and $V_{\xi, \overline{\xi}}$ are the two-dimensional irreducible representations of $D_n$ spanned by the eigenvectors for the normal subgroup $\mathbb{Z}_n \trianglelefteq D_n$ with eigenvalues $\xi, \overline{\xi}$. 
				\par
		
		\item
		Let $n=2k$ be an even number. Then $\mathbf{k}[R_n]\simeq V_{triv,even}\oplus V_{triv,odd}\oplus\underset{\xi\in \mu_k}{\bigoplus} V_{\xi, \overline{\xi}}^{\oplus 2}$  where the r.h.s. consists of simple right ideals and $V_{\xi, \overline{\xi}}$ are the two-dimensional irreducible representations of $D_k$ spanned by the eigenvectors for the normal subgroup $\mathbb{Z}_k \trianglelefteq D_k$ with eigenvalues $\xi, \overline{\xi}$. The representations  $V_{triv,even}$ and $V_{triv,odd}$ are the trivial one-dimensional representations corresponding to the orbits $X_{even}$ and $X_{odd}$ defined in the proof below.
		
\end{enumerate}

\end{pro}
\begin{proof}
We start with the case of odd $n$. Notice that $Inn(R_n)$ is a subgroup of $S_n$ isomorphic to the dihedral group $D_n$. Moreover, from multiplication table of $R_n$ we notice that multiplication by $e_i$ on the right correspond to reflection with respect to the line through the vertex  $i$ and the midpoint of the opposite edge. Thus finding the decomposition of $\mathbf{k}[R_n]$ into the sum of  where the r.h.s. consists of simple right ideals is equivalent to decomposing the representation $V=\{v_1,\hdots v_n\}$ given by $g\cdot v_i=v_{g\cdot i}$, where the action on the r.h.s. corresponds to the action of $Inn(R_n)$ on the $i$th vertex of a regular $n$-gon.
\par
Now we consider the case $n=2k$. Let $X_{odd}:=\{e_i\in X\mid i \; \mbox{ is odd}\}$ and $X_{even}:=\{e_i\in X\mid i \; \mbox{ is even}\}$. Note that $X=X_{even}\amalg X_{odd}$ with $|X_{even}|=|X_{odd}|=k$ and $e_i\rt e_j=e_i\rt e_{j+k}$ for any $i,j \in \{1,\hdots, n\}$. 
 The statement for $n=2k$ follows from the observation that the dihedral group of order $k$ (generated by $\varphi(e_1),\hdots,\varphi(e_k) \in S_n$)  acts on each orbit $X_{even}$ and  $X_{odd}$ the same way as described in the case of odd $n$. 
\end{proof}
\par

\par
\begin{rmk}
Let $X=$ Conj$(G)$ be the conjugation quandle on a group $G$, then the problem of decomposition of $\mathbf{k}[X]$ into indecomposable right $\mathbf{k}[X]$-modules is equivalent to decomposing $\mathbf{k}[G]$ with conjugation action into indecomposable representations. This was studied, see for example \cite{Roth, Solomon1, Solomon2}.
\end{rmk}

\section{On isomorphisms of quandle rings.}\label{iso}	
In this section, we investigate the problem of isomorphisms of quandle rings.  We introduce the notion of partition-type of quandles and show that if the quandle rings $\mathbf{k}[X]$ and $\mathbf{k}[Y]$ are isomorphic and the quandles $X$ and $Y$ are orbit $2$-transitive, then $X$ and $Y$ are of the same partition type.  First, we start with the following definition.
	
\begin{df}
Let $X$ be a finite quandle of cardinality $n$. The \textit{partition type} of $X$ is  $\lambda = (\lambda_1, \hdots, \lambda_n)$ with  $\lambda_j$ being the number of orbits of cardinality $j$ in $X$. 
\end{df}	

\begin{ex}
The partition type of the quandle $X=\{e_1,e_2\}\amalg\{e_3,e_4,e_5\}\amalg\{e_6,e_7\}\amalg\{e_8\}$ is $\lambda=1,2,1,0,0,\hdots$.
\end{ex}
\begin{thm}
\label{ISO}
Assume $char(\mathbf{k})\neq 2,3$. If the quandle rings $\mathbf{k}[X]$ and $\mathbf{k}[Y]$ are isomorphic and the quandles $X$ and $Y$ are orbit $2$-transitive  ($G_{X_i}$'s are not among the groups from $(i)-(v)$ in Theorem~\ref{TransClassif} in case $char(\mathbf{k})=p>3$), then $X$ and $Y$ are of the same partition type.
\end{thm}
\begin{proof}
Let $s$ stand for the number of orbits in $X$, $d$ for the number of orbits in $Y$ and $n$ be the number of elements in $X$ (and $Y$). The partition types of $X$ and $Y$ will be denoted by $\lambda$ and $\mu$. The number of elements in $X_j$ will be denoted by $n_j$.  Corollary~\ref{RightModDecomp} implies $\mathbf{k}[X]\simeq \underset{i=1}{\overset{s}{\bigoplus}}(V^i_{st}\oplus V^i_{triv})$ as the sum of right simple ideals, similarly, $\mathbf{k}[Y]\simeq \underset{j=1}{\overset{d}{\bigoplus}}(W^j_{st}\oplus W^j_{triv})$. Notice that $V^i_{st}=0$, if $|X_i|=0$. The isomorphism $\varphi: \mathbf{k}[X] \rightarrow \mathbf{k}[Y]$ induces another decomposition $\mathbf{k}[Y]=\overset{s}{\underset{j=1}{\oplus}}\varphi(V^i_{st})\oplus \varphi(V^i_{triv})$ with each summand being a simple ideal in $\mathbf{k}[Y]$. The Krull-Schmidt theorem asserts that the decomposition $\mathbf{k}[Y]\simeq \underset{j=1}{\overset{d}{\bigoplus}}(W^j_{st}\oplus W^j_{triv})$ is unique up to permutation of summands, from which we conclude that $\lambda=\mu$.
\end{proof}
\begin{cor}\label{WNPF1}
Let $X$ be a quandle of order 3. Then the three quandle rings arising from $X$ are not pairwise isomorphic. 
\end{cor}
\begin{cor}
The ring $\mathbf{k}[X]$ with $X$ the trivial quandle is not isomorphic to the quandle ring of any other orbit $2$-transitive quandle. 	
\end{cor}
\begin{rmk}
Actually, the ring $\mathbf{k}[X]$ with $X$ the trivial quandle is not isomorphic to the quandle ring of any other quandle. Indeed, if $\varphi: \mathbf{k}[X]\rightarrow \mathbf{k}[Y]$ is an isomorphism of quandle rings, it induces an isomorphism of augmentation ideals $\tilde{\varphi}: I_X\rightarrow I_Y$. In particular, this implies that $\mathbf{k}[X]$ for the trivial quandle $X$ is not isomorphic to $\mathbf{k}[Y]$ for any other quandle $Y$ as $a\cdot I_X=0$ for any $a\in \mathbf{k}[X]$, but this property does not hold in $\mathbf{k}[Y]$.
\end{rmk}

\begin{pro}\label{WNPF2}
Let $\mathbf{k}=\mathbb{Z}_p$, where $p$ is prime, and $X$ a quandle of order 3. Then the three quandle rings arising from $X$ are not pairwise isomorphic. 
\end{pro}

\begin{proof}
We count the number of zero columns in the multiplication table of each quandle ring and show that they are different, which implies the rings are not isomorphic.

\textbf{Case 1.} Let $X$ be $T_3$, the trivial quandle.

Let $\alpha_je_1+\beta_je_2+\gamma_je_3$ be an element of $\mathbb{Z}_p[T_3]$ and consider $a_{11}e_1+a_{12}e_2+a_{13}e_3\in \mathbb{Z}_p[T_3]$. Then
$$(\alpha_je_1+\beta_je_2+\gamma_je_3)\cdot (a_{11}e_1+a_{12}e_2+a_{13}e_3)=0\,\,\textnormal{for all}\,\, j$$ 
implies 
$$a_{11}+a_{12}+a_{13}=0.$$ 
So $a_{11}=-a_{12}-a_{13}$. Hence there are $p^2$ zero columns. 

\textbf{Case 2.}  Now let $X=\{1,2\}\ \amalg \{3\}$, the quandle with two orbits.

Let $\alpha_je_1+\beta_je_2+\gamma_je_3$ be an element of $\mathbb{Z}_p[X]$ and consider $a_{11}e_1+a_{12}e_2+a_{13}e_3\in \mathbb{Z}_p[X]$. Then
$$(\alpha_je_1+\beta_je_2+\gamma_je_3)\cdot (a_{11}e_1+a_{12}e_2+a_{13}e_3)=0\,\,\textnormal{for all}\,\, j$$ 
implies 
$$\gamma_j(a_{11}+a_{12}+a_{13})=0,$$
$$\alpha_j(a_{11}+a_{12})e_1+\alpha_ja_{13}e_2=0,$$ 
and
$$\beta_j(a_{11}+a_{12})e_2+\beta_ja_{13}e_1=0.$$

Hence we have $a_{13}=0$ and $a_{11}=-a_{12}$ which implies there are $p$ number of zero columns. 

\textbf{Case 3.}  Now $X=R_3$, the Takasaki quandle.

Let $\alpha_je_1+\beta_je_2+\gamma_je_3$ be an element of $\mathbb{Z}_p[R_3]$ and consider $a_{11}e_1+a_{12}e_2+a_{13}e_3 \in \mathbb{Z}_p[R_3]$. Then
$$(\alpha_je_1+\beta_je_2+\gamma_je_3)\cdot (a_{11}e_1+a_{12}e_2+a_{13}e_3)=0\,\,\textnormal{for all}\,\, j,$$ 
implies 
$$\alpha_j(a_{11}e_1+a_{12}e_3+a_{13}e_2)=0$$
$$\beta_j(a_{11}e_3+a_{12}e_2+a_{13}e_1)=0$$
$$\gamma_j(a_{11}e_2+a_{12}e_1+a_{13}e_3)=0$$
for all $j$. 

Since the system 
$$a_{11}e_1+a_{12}e_3+a_{13}e_2=0$$
$$a_{11}e_3+a_{12}e_2+a_{13}e_1=0$$
$$a_{11}e_2+a_{12}e_1+a_{13}e_3=0$$
is consistent, there is only one solution which is $a_{11}=a_{12}=a_{13}=0$. Thus there is only one zero column. 
\end{proof}

\begin{rmk}
Note that Corollary~\ref{WNPF1} works for all characteristic but 2 and 3, whereas Proposition~\ref{WNPF2} works for all prime characteristic.
\end{rmk}

Next we provide two examples, when the quandle rings $\mathbf{k}[X]$ and $\mathbf{k}[Y]$ are isomorphic, but the quandles $X$ and $Y$ are not, answering  \textit{Question}  $7.4$ of \cite{BPS}.
\begin{ex}
\label{CounterEx1}
Let $\mathbf{k}$ be a field with $char(\mathbf{k})=3$ and quandles $X$ and $Y$ be of cardinality $4$ with multiplication tables as below.

		$$(X,\rt)=\begin{array}{|c|c c c c|} 
		\hline
		\rt \ &  e_1 & e_2 & e_3 & e_4  \\ \hline
		\ e_1 &  e_1 & e_1 & e_2 & e_2  \\
		\ e_2 &  e_2 & e_2 & e_1 & e_1  \\
		\ e_3 &  e_3 & e_3 & e_3 & e_3  \\ 
		\ e_4 &  e_4 & e_4 & e_4 & e_4  \\ \hline
		\end{array}~~~~~~
	\textit{and}\; (Y,\rt)=\begin{array}{|c|c c c c|}
		\hline
		\ \rt &  e'_1 & e'_2 & e'_3 & e'_4  \\ \hline
		\ e'_1 &  e'_1 & e'_1 & e'_2 & e'_1  \\ 
		\ e'_2 &  e'_2 & e'_2 & e'_1 & e'_2  \\ 
		\ e'_3 &  e'_3 & e'_3 & e'_3 & e'_3  \\
		\ e'_4 &  e'_4 & e'_4 & e'_4 & e'_4  \\ \hline
		\end{array}$$
		
\par
The isomorphism  is given by $\varphi:\mathbf{k}[X]\overset{\sim}{\rightarrow}\mathbf{k}[Y]$, where $\varphi=\left(\begin{array}{cccc}
1 & 0 & 0 & 1\\
0 & 1 & 0 & 1\\
0 & 0 & 1 & 1\\
0 & 0 & 0 & 1\\
\end{array}\right)$.
\end{ex}

\par
Generalizing Example~\ref{CounterEx1}, we consider $\tilde{X}=X\amalg\{e_5\}\amalg\{e_6\}\amalg\hdots\amalg\{e_n\}$ and $\tilde{Y}=Y\amalg\{e_5\}\amalg\{e_6\}\amalg\hdots\amalg\{e_n\}$. Clearly the quandles $\tilde{X}$ and $\tilde{Y}$ are not isomorphic. Let $\mathbf{k}$ be a field with $char(\mathbf{k})=p$ with $p \mid n-1$. Then $\varphi:\mathbf{k}[X]\overset{\sim}{\rightarrow}\mathbf{k}[Y]$  given by  $\varphi(e_i) = e'_i $ for $i\neq 4$ and $\varphi(e_4)=\sum\limits_{j=1}^{n}e'_j$  is a ring isomorphism.

\begin{ex}
\label{CounterEx2}
Let $\mathbf{k}$ be a field with $char(\mathbf{k})=0$ and quandles $X$ and $Y$ be of cardinality $7$ with multiplication tables as below.

		$$(X,\rt)=\begin{array}{|c|c c c c c c c|} 
		\hline
		\rt \ &  e_1 & e_2 & e_3 & e_4 & e_5 & e_6 &  e_7 \\ \hline
		\ e_1 &  e_1 & e_1 & e_1 & e_1 & e_2 & e_2 & e_1\\
		\ e_2 &  e_2 & e_2 & e_2 & e_2 & e_1 & e_1 & e_2 \\
		\ e_3 &  e_3 & e_3 & e_3 & e_3 & e_3 & e_4 & e_3 \\ 
		\ e_4 &  e_4 & e_4 & e_4 & e_4 & e_4 & e_3 & e_4 \\
		\ e_5 &  e_5 & e_5 & e_5 & e_5 & e_5 & e_5 & e_5 \\
		\ e_6 &  e_6 & e_6 & e_6 & e_6 & e_6 & e_6 & e_6 \\
		\ e_7 &  e_7 & e_7 & e_7 & e_7 & e_7 & e_7 & e_7\\ \hline
		\end{array}~~~~~~
	\textit{and}\; (Y,\rt)=\begin{array}{|c|c c c c c c c|} 
			\hline
		\rt \ &  e_1 & e_2 & e_3 & e_4 & e_5 & e_6 &  e_7  \\ \hline
		\ e_1 &  e_1 & e_1 & e_1 & e_1 & e_2 & e_1 & e_1\\
		\ e_2 &  e_2 & e_2 & e_2 & e_2 & e_1 & e_2 & e_2 \\
		\ e_3 &  e_3 & e_3 & e_3 & e_3 & e_3 & e_4 & e_3 \\ 
		\ e_4 &  e_4 & e_4 & e_4 & e_4 & e_4 & e_3 & e_4 \\
		\ e_5 &  e_5 & e_5 & e_5 & e_5 & e_5 & e_5 & e_5  \\
		\ e_6 &  e_6 & e_6 & e_6 & e_6 & e_6 & e_6 & e_6 \\
		\ e_7 &  e_7 & e_7 & e_7 & e_7 & e_7 & e_7 & e_7  \\ \hline
		\end{array}$$ 
		
\par
One family of isomorphisms  is given by $\varphi:\mathbf{k}[X]\overset{\sim}{\rightarrow}\mathbf{k}[Y]$, where 
\[\varphi=\left(\begin{array}{ccccccc}
1 & 0 & 0 & 0 & 0 & 0 & 0\\
0 & 1 & 0 & 0 & 0 & 0 & 0\\
0 & 0 & 1 & 0 & 0 & 0 & 0\\
0 & 0 & 0 & 1 & 0 & 0 & 0\\
0 & 0 & 0 & 0 & 1 & 1 & 0\\
0 & 0 & 0 & 0 & 0 & 1 & 0\\
0 & 0 & 0 & 0 & 0 & -1 & 1\\
\end{array}\right).\]
\end{ex}

\par
This example can be generalized. The following definition and proposition can be found in Section $2$ of \cite{Nelson}.
\par

\begin{df}  Let $Q$ be a finite quandle. For any element $x \in Q$, let $c(x) = |\{y \in Q : y \rt x = y\}|$ and let $r(x) = |\{y \in Q : x \rt y = x\}|$ .
Then we define the quandle polynomial of $Q$, $qp_Q(s, t)$, to be $qp_Q(s, t) := \sum\limits_{x\in Q}s^{r(x)}t^{c(x)}$.
\end{df}
\begin{pro}
\label{IsoInvar}
If $Q$ and $Q'$ are isomorphic finite quandles, then $qp_Q(s, t) = qp_{Q'}(s, t)$.
\end{pro}
\par
In particular, for the quandles $X$ and $Y$ from Example \ref{CounterEx2}, we have $qp_X(s, t)=s^7(t^7+t^5+t^3)+2s^6t^7+2s^5t^7$  and $qp_Y(s, t)=t^7(4s^6+s^7)+2t^5s^7$, confirming that the two quandles are not isomorphic. Let $Z$ be any finite quandle, consider the quandles $X\amalg Z$ and $Y\amalg Z$ with the operations $x_1\rt x_2$, $y_1\rt y_2 $ and $z_1\rt z_2 $ replicating the ones in $X,Y$ and $Z$ respectively, while $x \rt z =x, z \rt x =z$ and $y \rt z =y, z \rt y =z$ for any $x, x_1, x_2 \in X, y, y_1, y_2 \in Y$ and $z, z_1, z_2 \in Z$. It is easy to see that $qp_{X\amalg Z}(s, t) \neq qp_{Y\amalg Z}(s, t)$, so, it follows from Proposition \ref{IsoInvar} that $X\amalg Z$ and $Y\amalg Z$ are not isomorphic. However, there are  isomorphisms $\mathbf{k}[X\amalg Z]\overset{\sim}{\rightarrow} \mathbf{k}[Y\amalg Z]$ given by the operators written in block form as $\varphi|_{\mathbf{k}[X]\rightarrow \mathbf{k}[Y]}$, $Id|_{ \mathbf{k}[Z]\rightarrow  \mathbf{k}[Z]}$ and $0$ on the remaining two blocks for $\varphi$ from Example \ref{CounterEx2}.
\par
These examples give negative answers to both \textit{Questions} $7.4$ and $7.5$ proposed in \cite{BPS}.  Indeed, it is not hard to see that Example~\ref{CounterEx1} provides a negative answer to \textit{Question $7.5$}. Let $I\subset  X$ and $\tilde{I}\subset Y$ be the augmentation ideals. Then $I^{\geq 2}=\tilde{I}^{\geq 2}=(e_1-e_2)$, hence, $I^{k}/I^{k+1}=\tilde{I}^{k}/\tilde{I}^{k+1}=0$ for $k \geq 2$. Also, $I^{0}/I^{1}\cong\tilde{I}^{0}/\tilde{I}^{1}$ and $I^{1}/I^{2}\cong\tilde{I}^{1}/\tilde{I}^{2}$ are $1$-dimensional and $2$-dimensional vector spaces over $\mathbf{k}$.
 Indeed, let $I\subset  X$ and $\tilde{I}\subset Y$ be the augmentation ideals. Then $I^{\geq 2}=\tilde{I}^{\geq 2}=(e_1-e_2)$, hence, $I^{k}/I^{k+1}=\tilde{I}^{k}/\tilde{I}^{k+1}=0$ for $k \geq 2$. Also, $I^{0}/I^{1}\cong\tilde{I}^{0}/\tilde{I}^{1}$ and $I^{1}/I^{2}\cong\tilde{I}^{1}/\tilde{I}^{2}$ are $1$-dimensional and $2$-dimensional vector spaces over $\mathbf{k}$.
\subsection{Parameter spaces for quandle ring morphisms}
 Let $\mathbf{k}[X]$ and $\mathbf{k}[Y]$ be two quandle rings with $|X|=|Y|=n$. Then $\varphi=\left(\begin{array}{ccc}
a_{11} & \hdots & a_{n1} \\
\vdots & \ddots & \vdots \\
a_{1n} & \hdots & a_{nn}
\end{array}\right)$ defines a map $\mathbf{k}[X]\rightarrow\mathbf{k}[Y]$ if $\varphi(e_ie_j)=\varphi(e_i)\varphi(e_j)$ for all $i,j \in \{1,\hdots,n\}$. This in turn produces $n^3$ quadratic equations in the $n^2$-dimensional vector space of parameters ($a_{ij}$'s). Furthermore, $\varphi$ is an isomorphism if the corresponding matrix is of full rank. We provide one possible application. 
\par
\begin{ex}
 Let $X_1,X_2$ and $X_3$ be quandles of cardinality one. We show that the rings $\mathbf{k}[X_1]\oplus\mathbf{k}[X_2]\oplus\mathbf{k}[X_3]$ and $\mathbf{k}[X_1\amalg X_2\amalg X_3]$ are not isomorphic. Indeed the conditions $\varphi^2(e_i)=\varphi(e_i)$ and $\varphi(e_ie_j)=\varphi(e_i)\varphi(e_j)=0$ give rise to the equations 
\begin{equation*}
 \begin{cases} a_{11}=a_{11}(a_{11}+a_{12}+a_{13}) \\ a_{12}=a_{12}(a_{11}+a_{12}+a_{13})  \\a_{13}=a_{13}(a_{11}+a_{12}+a_{13}) \end{cases}, \\ \\ \begin{cases} a_{21}=a_{21}(a_{21}+a_{22}+a_{23}) \\ a_{22}=a_{22}(a_{21}+a_{22}+a_{23})  \\a_{23}=a_{23}(a_{21}+a_{22}+a_{23}) \end{cases},
 \end{equation*}
  \begin{equation*}
  \begin{cases} a_{31}=a_{31}(a_{31}+a_{32}+a_{33}) \\ a_{32}=a_{32}(a_{31}+a_{32}+a_{33})  \\a_{33}=a_{33}(a_{31}+a_{32}+a_{33}) \end{cases}, \\ \\ \begin{cases} 0=a_{11}(a_{21}+a_{22}+a_{23}) \\ 0=a_{12}(a_{21}+a_{22}+a_{23})  \\0=a_{13}(a_{21}+a_{22}+a_{23}) \end{cases}, \\ \hdots 
  \end{equation*}
\par
We show that the system is already inconsistent. As $rk(\varphi)=3$, to satisfy the first nine equations, we must have $a_{11}+a_{12}+a_{13}=a_{21}+a_{22}+a_{23}=a_{31}+a_{32}+a_{33}=1$, however, this leaves the only possibility $a_{11}=a_{12}=a_{13}=0$ for the remaining three equations, which contradicts the assumption on the rank of  $\varphi$.	
\end{ex}

\section{A proof of conjecture 6.5 in \cite{BPS} } \label{augmentation}

The goal of this section is to give a solution to \textit{conjecture 6.5} in \cite{BPS} concerning the quotient of the powers of the augmentation ideal of the quandle ring of dihedral quandles.

In this Section, we confirm one of the conjectures suggested in \cite{BPS} and present the progress of another conjecture.   First we state the conjecture as presented in \cite{BPS}. 
\begin{conj}
Let $R_n$ be the dihedral quandle.
\begin{enumerate}
\item If $n>1$ is an odd integer, then $\Delta^k(R_n) / \Delta^{k+1}(R_n) \cong \mathbb{Z}_n$ for all $k \geq 1$.
\item If $n>2$ is an even integer, then $|\Delta^k(R_n) / \Delta^{k+1}(R_n)| = n$ for all $k \geq 2$.
\end{enumerate}
\end{conj}
\bigskip

We prove the first part of the above conjecture. We slightly abuse the notation and write $e_i$ for $[e_i]$. 

\begin{thm}
Let $n$ be odd. Then $\Delta^k(R_n) / \Delta^{k+1}(R_n) \cong \mathbb{Z}_n$ for all $k \geq 1$. 
\end{thm}

\begin{proof} We prove the result by induction on $k$. 
First we show that $\Delta(R_n)/\Delta^2(R_n)$ is generated as an abelian group by $e_1$ and $\Delta(R_n) / \Delta^2(R_n) \cong \mathbb{Z}_n$.
Consider the integral quandle ring of the dihedral quandle $R_n= \{a_0, a_1, a_2, \ldots, a_{n-1}\}$. Let $e_1= a_1-a_0,\,e_2=a_2-a_0,\,\ldots , e_{n-1}=a_{n-1}-a_0$.
Then  $\Delta(R_n)= \langle e_1, e_2, \ldots, e_{n-1} \rangle $. 
The abelian group $\Delta^2(R_n)$ is generated by the products $e_i  \rt e_j$.  

It is clear that $e_{2i}=-e_{n-2i}$, where $1\leq i\leq \frac{n-1}{2}$. In particular, when $i=\frac{n-1}{2}$ we have $e_1=-e_{n-1}$. Since $e_1=-e_{n-1}$, we have the following.

$e_1\rt e_1=0$ implies $e_2=2e_1$, and

$e_{n-1}\rt e_1=0$ implies $e_3=3e_1$. 

Continuing in this manner in the first column from bottom to top, we get $ne_1=0$. Hence $\Delta/\Delta^2(R_n)$ is generated as an abelian group by $e_1$ and $\Delta(R_n) / \Delta^2(R_n) \cong \mathbb{Z}_n$. 

Now assume that $\Delta^k(R_n) / \Delta^{k+1}(R_n) \cong \mathbb{Z}_n$. We show that $\Delta^{k+1}(R_n) / \Delta^{k+2}(R_n) \cong \mathbb{Z}_n$. 

Let $[\alpha]$ be a generator of $\Delta^k(R_n) / \Delta^{k+1}(R_n)$, i.e. $\Delta^k(R_n) / \Delta^{k+1}(R_n)=\{\alpha, 2\alpha, 3\alpha, \ldots, (n-1)\alpha, n\alpha\}$. Note that the elements $e_i\alpha$, where $1\leq i\leq n-1$, generate $\Delta^{k+1}(R_n) / \Delta^{k+2}(R_n)$. Since $e_k=ke_1$ for $1\leq k\leq n-1$, the above set is generated by $e_1\alpha$. Recall that $ne_1=0$. Thus we have
$$\Delta^{k+1}(R_n) / \Delta^{k+2}(R_n) \cong \mathbb{Z}_n.$$ 

\end{proof}

Now we present a generalization of \cite[Proposition~6.3 (1)]{BPS} where $n>2$ is even. We will use the following equation (see \eqref{neranga1} in Case 2 in appendix) in the proof of Theorem~\ref{neranga3}.

\begin{equation}\label{neranga2}
e_{i} \rt e_{1}=-e_{2}-e_{n-i}+e_{n-i+2},\,\,\,\,\textnormal{for}\,\, 3\leq i\leq n-3.
\end{equation}

\begin{thm}\label{neranga3}
Let $n=2k$ for some positive integer $k$. Then $\Delta(\R_n) / \Delta^2(\R_n)\cong \mathbb{Z}\oplus \mathbb{Z}_{k}$. 
\end{thm}

\begin{proof}
We adhere to the following strategy: 
\begin{enumerate}
\item [(1)] 
Show that $\Delta(R_n) / \Delta^{2}(R_n)$ is generated by the classes of $e_1$ and $e_2$
\item [(2)] 
 Verify that $e_{2s}=se_2$ for $2\leq s\leq k-1$. It follows that the abelian subgroup generated by $e_2$ is $\mathbb{Z}_k$. 
 \item [(3)] 
 Check that the abelian subgroup generated by $e_1$ has no torsion.
 \end{enumerate}
 \par
 
First we show that (1) holds. For this we claim that
 
\begin{equation}\label{Star}
e_l = \left\{
        \begin{array}{ll}
           \frac{l}{2}\,e_2 &  \textnormal{if}\,\,l\,\,\textnormal{is even},\\[0.3cm]
            \lfloor \frac{l}{2} \rfloor\,e_2+e_1 &  \textnormal{if}\,\,l\,\,\textnormal{is odd}.
        \end{array}
    \right.
\end{equation}
 
Let $l>3$ be even. We prove by induction on $l$. 

Let $l=4$. Then we have $e_4=2e_2$ which is true since $0=e_{n-2} \rt e_1=-2e_2+e_4$. Assume that it is true for $l$ and consider $l+1$. Let $i=n-l$ in \eqref{neranga2} (consider the first column in table similarly in the case  $n\equiv 0\pmod{4}$). 

Then we have 
$$e_{n-l} \rt e_{1}=-e_{2}-e_{l}+e_{l+2}.$$

$e_{n-l} \rt e_{1}=0$ implies $e_{l+2}=e_l+e_2$. From inductive assumption we have
$$e_{l+2}=\frac{l}{2}\,e_2+e_2=\Big(\frac{l+2}{2}\Big)\,e_2.$$

Let $l$ be odd. The by Division Algorithm we have
$$l=\lfloor \frac{l}{2} \rfloor\,2+1,$$
which implies $l-1$ is even. Then we have $e_{l-1}=e_{\lfloor \frac{l}{2} \rfloor\,\cdot 2}=\lfloor \frac{l}{2} \rfloor\,e_2$.

We show that for $3\leq l\leq n-1$ we have
$$e_l=e_{l-1}+e_1$$
by induction on $l$. 

When $l=3$, we have $e_3=e_2+e_1$ which is true since $0=e_{n-1} \rt e_1=-e_1-e_2+e_3=0$.

Assume that it is true for $l$. When $l+2$, let $i=n-l$ in \eqref{neranga2} (consider the first column in table similarly in the case  $n\equiv 0\pmod{4}$). 

Then we have 
$$e_{n-l} \rt e_{1}=-e_{2}-e_{l}+e_{l+2}.$$

$e_{n-l} \rt e_{1}=0$ implies $e_{l+2}=e_l+e_2$. From inductive assumption we have
$$e_{l+2}=e_{l-1}+e_1+e_2=\Big(\frac{l-1}{2}\Big)\,e_2+e_1+e_2=\Big(\frac{l+1}{2}\Big)\,e_2+e_1=\lfloor \frac{l+2}{2} \rfloor\,e_2+e_1.$$

This completes the claim. Now consider

$e_i \rt e_j=(a_i-a_0)(a_j-a_0)=a_{2j-i}-a_{n-i}-a_{2j}+a_0=-e_{2j-i}+e_{n-i}+e_{2j}$.

Using the above claim we show that $\Delta^2(R_n)$, which consists of $e_i\rt e_j$, for all $1\leq i,j\leq n-1$, is generated by relations from \eqref{Star}. This proves that the abelian group generated by $e_1$ is torsion free. 

Let $i$ be even. 

\[
\begin{split}
e_i \rt e_j&=-e_{2j-i}+e_{n-i}+e_{2j}\cr
&=-\Big(\frac{2j-i}{2}\Big)\,e_2+\Big(\frac{n-i}{2}\Big)\,e_2+\Big(\frac{2j}{2}\Big)\,e_2\cr
&=\frac{n}{2}\,e_2\cr
&=ke_2\cr
&=0.
\end{split}
\]

Now let $i$ be odd. 

\[
\begin{split}
e_i \rt e_j&=-e_{2j-i}+e_{n-i}+e_{2j}\cr
&=-\lfloor \frac{2j-i}{2}\rfloor \,e_2-e_1+\lfloor \frac{n-i}{2}\rfloor\,e_2+e_1+\Big(\frac{2j}{2}\Big)\,e_2\cr
&=-\lfloor \frac{2j-i}{2}\rfloor \,e_2+\lfloor \frac{n-i}{2}\rfloor\,e_2+j\,e_2\cr
&=-\Big( \frac{2j-i-1}{2}\Big) \,e_2+\Big( \frac{n-i-1}{2}\Big)\,e_2+j\,e_2\cr
&=\frac{n}{2}\,e_2\cr
&=ke_2\cr
&=0.
\end{split}
\]
 
Therefore we have 
 $$\Delta(R_n) / \Delta^2(R_n)\cong \mathbb{Z}\oplus \mathbb{Z}_{k}.$$
 
\end{proof}

\noindent
\textbf {Acknowledgement:}
The authors would like to thank Valeriy G. Bardakov and Mahender Singh for their comments which improved the first version of the paper. We'd also like to thank the referee for useful comments. Thanks are also due to Anthony Curtis for help with computer programming and Mohamed Elbehiry for the invaluable discussions.

%
%
\clearpage

\section*{Appendix}
  
Here we present some patterns in multiplication tables for $\Delta(R_n)$ considering two cases: $n\equiv 0\pmod{4}$ and $n\equiv 2\pmod{4}$. We also give two examples to elaborate following equations. 

\textbf{Case 1.} $n\equiv 0\pmod{4}$. 

Consider the integral quandle ring of the dihedral quandle $R_n= \{a_0, a_1, a_2, \ldots, a_{n-1}\}$. Let $e_1= a_1-a_0,\,e_2=a_2-a_0,\,\ldots , e_{n-1}=a_{n-1}-a_0$.
Then  $\Delta(R_n)= \langle e_1, e_2, \ldots, e_{n-1} \rangle $. To determine $\Delta^2(R_n)$, we compute the products $e_i \rt e_j$. Then we have the following. 

$$e_{2i} \rt e_i=-e_{2i}-e_{n-2i},\,\,\,\,\textnormal{for}\,\,1\leq i\leq \frac{n}{4}$$
$$e_{n-2i} \rt e_i=-2e_{2i}+e_{4i},\,\,\,\,\textnormal{for}\,\,1\leq i\leq \frac{n}{4}-1$$
$$e_{2i} \rt e_{\frac{n}{2}-i}=e_{n-4i}-2e_{n-2i},\,\,\,\,\textnormal{for}\,\,1\leq i\leq \frac{n}{4}-1$$
$$e_{i} \rt e_{\frac{n}{4}}=-e_{n-i}-e_{\frac{n}{2}}+e_{\frac{n}{2}+n-i},\,\,\,\,\textnormal{for}\,\,\frac{n}{2}+1\leq i\leq n-1$$
$$e_{i} \rt e_{\frac{n}{4}}=e_{\frac{n}{2}-i}-e_{\frac{n}{2}}-e_{n-i},\,\,\,\,\textnormal{for}\,\,\,1\leq i\leq \frac{n}{2}-1$$
$$e_{i}\rt e_{\frac{n}{2}}=0,\,\,\,\,\textnormal{for all}\,\,\,i.$$

We give an example when $n=8$. 

\begin{center}
\tiny
\begin{tabular}{|c||c|c|c|c|c|c|c|}
  \hline
 $\cdot$ & $e_1$ & $e_2$ & $e_3$ & $e_4$ & $e_5$ & $e_6$ & $e_7$\\
   \hline
     \hline
$e_1$ & $$ & $e_3-e_4-e_7$ & $$ & $0$ & $$ & $$ & $$\\
  \hline
$e_2$  & $-e_2-e_6$ & $e_2-e_4-e_6$ & $e_4-2e_6$ & $0$ & $-e_2-e_6$ & $$ & $e_4-2e_6$\\
  \hline
$e_3$ & $$ & $e_1-e_4-e_5$ & $$ & $0$ & $$ & $$ & $$\\
  \hline
$e_4$ & $$ & $-2e_4$ & $$ & $0$ & $$ & $-2e_4$ & $$\\
  \hline
$e_5$ & $$ & $-e_3-e_4+e_7$ & $$ & $0$ & $$ & $$ & $$\\
  \hline
$e_6$ & $-2e_2+e_4$ & $-e_2-e_4+e_6$ & $-e_2-e_6$ & $0$ & $-2e_2+e_4$ & $$ & $-e_2-e_6$\\
  \hline
$e_7$ & $$ & $-e_1-e_4+e_5$ & $$ & $0$ & $$ & $$ & $$\\
  \hline
\end{tabular}\\
\end{center}

Note that $i$-th column $=$ $(\frac{n}{2}+i)$th column. 

\textbf{Case 2.} $n\equiv 2\pmod{4}$. 

Similar computations to Case 1 yield the following. 

$$e_{2i}\rt e_i=-e_{2i}-e_{n-2i},\,\,\,\,\textnormal{for}\,\,1\leq i\leq \lfloor \frac{n}{4} \rfloor$$
$$e_{n-2i}\rt e_i=-2e_{2i}+e_{4i},\,\,\,\,\textnormal{for}\,\,1\leq i\leq \lfloor \frac{n}{4} \rfloor$$
$$e_{2i}\rt e_{\frac{n}{2}-i}=e_{n-4i}-2e_{n-2i},\,\,\,\,\textnormal{for}\,\,1\leq i\leq \lfloor \frac{n}{4} \rfloor$$
$$e_{i}\rt e_{\frac{n}{2}}=0,\,\,\,\,\textnormal{for all}\,\,\,i$$
$$e_{1}\rt e_{1}=e_1-e_2-e_{n-1}$$
$$e_{n-1}\rt e_{1}=-e_1-e_2+e_{3}$$
\begin{equation}\label{neranga1}
e_{i}\rt e_{1}=-e_{2}-e_{n-i}+e_{n-i+2},\,\,\,\,\textnormal{for}\,\, 3\leq i\leq n-3.
\end{equation}

We give an example when $n=10$. 

\begin{center}
\tiny
\begin{tabular}{|c||c|c|c|c|c|c|c|c|c|}
  \hline
 $\cdot$ & $e_1$ & $e_2$ & $e_3$ & $e_4$ & $e_5$ & $e_6$ & $e_7$ & $e_8$ & $e_9$\\
   \hline
$e_1$ & $e_1-e_2-e_9$ & $$ & $$ & $$ & $0$ & $$  & $$ & $$ & $$ \\
  \hline
$e_2$  & $-e_2-e_8$ & $$ & $$ & $e_6-2e_8$ & $0$ & $-e_2-e_8$ & $$  & $$ & $e_6-2e_8$\\
  \hline
$e_3$ & $-e_2-e_7+e_9$ & $$ & $$ & $$ & $0$ & $$ & $$ & $$ & $$ \\
  \hline
$e_4$ & $-e_2-e_6+e_8$ & $-e_4-e_6$ & $e_2-2e_6$ & $$  & $0$ & $$ & $-e_4-e_6$ & $e_2-2e_6$ & $$ \\
  \hline
$e_5$ & $-e_2-e_5+e_7$ & $$ & $$ & $$ & $0$ & $$ & $$ & $$ & $$\\
  \hline
  $e_6$ & $-e_2-e_4+e_6$ & $-2e_4+e_8$ & $-e_4-e_6$ & $$ & $0$ & $$ & $-2e_4+e_8$ & $-e_4-e_6$ & $$\\
  \hline
  $e_7$ & $-e_2-e_3+e_5$ & $$ & $$ & $$ & $0$ & $$ & $$ & $$ & $$\\
  \hline
  $e_8$ & $-2e_2+e_4$ & $$ & $$ & $-e_2-e_8$ & $0$ & $-2e_2+e_4$ & $$ & $$ & $-e_2-e_8$\\
  \hline
  $e_9$ & $-e_1-e_2+e_3$ & $$ & $$ & $$ & $0$ & $$ & $$ & $$ & $$\\
  \hline
\end{tabular}\\
\end{center}

Note that $i$-th column $=$ $(\frac{n}{2}+i)$th column. 
\end{document}